\documentclass[reqno,a4paper]{amsart}
\usepackage{amssymb,graphicx,setspace}
\usepackage{ifpdf}
\ifpdf
 \usepackage[hyperindex]{hyperref}
\else
 \expandafter\ifx\csname dvipdfm\endcsname\relax
 \usepackage[hypertex,hyperindex]{hyperref}%
 \else
 \usepackage[dvipdfm,hyperindex]{hyperref}%
 \fi
\fi
\theoremstyle{plain}
\newtheorem{thm}{Theorem}[section]
\newtheorem{lem}{Lemma}[section]

\newtheorem{prop}{Proposition}[section]
\newtheorem{conj}{Conjecture}[section]
\newtheorem{open}{Problem}[section]
\theoremstyle{definition}
\newtheorem{dfn}{Definition}[section]
\theoremstyle{remark}
\newtheorem{rem}{Remark}[section]
\allowdisplaybreaks[4]
\DeclareMathOperator{\td}{d\mspace{-2mu}}
\DeclareMathOperator{\diag}{diag}
\newcommand{\cmdeg}[1]{\sideset{}{_\mathrm{cm}^{#1}}\deg}
\newcommand{\cbdeg}[1]{\sideset{}{_\mathrm{cb}^{#1}}\deg}
\newcommand{\stdeg}[1]{\sideset{}{_\mathrm{s}^{#1}}\deg}
\newcommand{\opdeg}[1]{\sideset{}{_\mathrm{op}^{#1}}\deg}
\numberwithin{equation}{section}

\begin{document}

\title[Integral representations of some functions]
{Integral representations and properties of some functions involving the logarithmic function}

\author[F. Qi]{Feng Qi}
\address[Qi]{Department of Mathematics, School of Science, Tianjin Polytechnic University, Tianjin City, 300387, China; Institute of Mathematics, Henan Polytechnic University, Jiaozuo City, Henan Province, 454010, China}
\email{\href{mailto: F. Qi <qifeng618@gmail.com>}{qifeng618@gmail.com}, \href{mailto: F. Qi <qifeng618@hotmail.com>}{qifeng618@hotmail.com}, \href{mailto: F. Qi <qifeng618@qq.com>}{qifeng618@qq.com}}
\urladdr{\url{http://qifeng618.wordpress.com}}

\author[W.-H. Li]{Wen-Hui Li}
\address[Li]{Department of Mathematics, School of Science, Tianjin Polytechnic University, Tianjin City, 300387, China}
\email{\href{mailto: W.-H. Li <wen.hui.li@foxmail.com>}{wen.hui.li@foxmail.com}, \href{mailto: W.-H. Li <wen.hui.li102@gmail.com>}{wen.hui.li102@gmail.com}}

\begin{abstract}
By using Cauchy integral formula in the theory of complex functions, the authors establish some integral representations for the principal branches of several complex functions involving the logarithmic function, find some properties, such as being operator monotone function, being complete Bernstein function, and being Stieltjes function, for these functions, and verify a conjecture on complete monotonicity of a function involving the logarithmic function.
\end{abstract}



\thanks{This paper was typeset using \AmS-\LaTeX}

\maketitle

\section{Preliminaries}

We recall some definitions, notion, and characterizations.

\begin{dfn}[{\cite[Chapter~IV]{Widder-Laplace-Transform-41}}]
An infinitely differentiable function $f$ on an interval $I$ is said to be completely monotonic on $I$ if it satisfies
\begin{equation*}
(-1)^{n-1}f^{(n-1)}(t)\ge0
\end{equation*}
for $x \in I$ and $n\in\mathbb{N}$, where $\mathbb{N}$ stands for the set of all positive integers.
\end{dfn}

For our own convenience, we denote the class of all completely monotonic functions on an interval $I$ by the notation $\mathcal{C}[I]$.
The class $\mathcal{C}[(0,\infty)]$ may be characterized by

\begin{prop}[{\cite[Theorem~12b]{Widder-Laplace-Transform-41}}]
A necessary and sufficient condition that $f(x)$ should be completely monotonic for $0<x<\infty$ is that
\begin{equation} \label{berstein-1}
f(x)=\int_0^\infty e^{-xt}\td\alpha(t),
\end{equation}
where $\alpha(t)$ is non-decreasing and the integral converges for $0<x<\infty$.
\end{prop}

\begin{dfn}[\cite{psi-proper-fraction-degree-two.tex, simp-exp-degree-ext-gjma.tex, Qi-Zhang-Li-MJM-14}]
Let $f(x)$ be a completely monotonic function on $(0,\infty)$ and denote $f(\infty)=\lim_{x\to\infty}f(x)$. If for some $r\in\mathbb{R}$ the function
\begin{equation*}
x^r[f(x)-f(\infty)]
\end{equation*}
is completely monotonic on $(0,\infty)$ but
\begin{equation*}
x^{r+\varepsilon}[f(x)-f(\infty)]
\end{equation*}
is not for any positive number $\varepsilon>0$, then we say that the number $r$ is the completely monotonic degree of $f(x)$ with respect to $x\in(0,\infty)$; if for all $r\in\mathbb{R}$ each and every $x^r[f(x)-f(\infty)]$ is completely monotonic on $(0,\infty)$, then we say that the completely monotonic degree of $f(x)$ with respect to $x\in(0,\infty)$ is $\infty$.
\end{dfn}

\begin{rem}
It was proved in~\cite[Section~1.5]{Bessel-ineq-Dgree-CM.tex} that the completely monotonic degree $\cmdeg{x}[f(x)]$ equals $\infty$ if and only if $f(x)$ is nonnegative and identically constant.
\end{rem}

For convenience and simplicity, the notation
\begin{equation*}
\cmdeg{x}[f(x)]
\end{equation*}
was designed in~\cite[p.~9890]{psi-proper-fraction-degree-two.tex} to denote the completely monotonic degree of $f(x)$ with respect to $x\in(0,\infty)$. Completely monotonic functions on $(0,\infty)$ of degree $r\ge0$ with respect to $x$ can be characterized by~\cite[Remark~1.6]{Koumandos-Pedersen-JMAA-09} which may be reformulated as follows.

\begin{prop}
Let $f(x)$ be a completely monotonic function on $(0,\infty)$ and
\begin{equation}\label{egamma}
\Gamma(z)=\int^\infty_0t^{z-1} e^{-t}\td t, \quad \Re(z)>0
\end{equation}
stand for the classical Euler's gamma function. Then $\cmdeg{t}[f(t)]=r\ge0$ if and only if
\begin{equation}
f(x)=\int_0^\infty\biggl[\frac1{\Gamma(\alpha)} \int_0^s(s-\tau)^{\alpha-1}\td\mu_\alpha(\tau)\biggr]e^{-xs}\td s
\end{equation}
converges for all $0\le\alpha\le r$ and $0<x<\infty$, where $\mu_\alpha(\tau)$ is a family of non-negative
measures on $(0,\infty)$ if and only if $0\le\alpha\le r$.
\end{prop}

\begin{dfn}[\cite{Atanassov, CBerg, absolute-mon-simp.tex, compmon2, minus-one, subadditive-qi-guo-jcam.tex, SCM-2012-0142.tex, Open-TJM-2003-Banach.tex}]
An infinitely differentiable and positive function $f$ is said to be logarithmically completely monotonic on an interval $I$ if
\begin{equation*}
(-1)^k[\ln f(x)]^{(k)}\ge0
\end{equation*}
hold on $I$ for all $k\in\mathbb{N}$.
\end{dfn}

\begin{dfn}[{\cite[Definition~2.1]{Schilling-Song-Vondracek-2nd}}]
If a function $f:(0,\infty)\to[0,\infty)$ can be written in the form
\begin{equation}\label{dfn-stieltjes}
f(x)=\frac{a}x+b+\int_0^\infty\frac1{s+x}{\td\mu(s)},
\end{equation}
where $a,b$ are non-negative constants and $\mu$ is a measure on $(0,\infty)$ such that $\int_0^\infty\frac1{1+s}\td\mu(s)<\infty$, then we say that $f$ is a Stieltjes function.
\end{dfn}

\begin{prop}[\cite{CBerg, absolute-mon-simp.tex, compmon2, minus-one}]
The inclusions
\begin{equation}\label{S-L-C-relation}
\mathcal{L}[I]\subset\mathcal{C}[I] \quad \text{and}\quad \mathcal{S}\setminus\{0\}\subset\mathcal{L}[(0,\infty)]
\end{equation}
are valid, where $\mathcal{S}$, $\mathcal{L}[I]$, and $\mathcal{C}[I]$ denote respectively the set of all Stieltjes functions, the set of all logarithmically completely monotonic functions on an interval $I$, and the set of all completely monotonic functions on $I$.
\end{prop}

\begin{dfn}[\cite{Schilling-Song-Vondracek-2nd}]
An infinitely differentiable function $f:I\to[0,\infty)$ is called a Bernstein function on an interval $I$ if $f'(t)$ is completely monotonic on $I$.
\end{dfn}

We denote the group of all Bernstein functions on an interval $I$ by $\mathcal{B}[I]$.

\begin{prop}[{\cite[Theorem~3.2]{Schilling-Song-Vondracek-2nd}}]
A function $f:(0,\infty)\to[0,\infty)$ is a Bernstein function if and only if it admits the representation
\begin{equation}\label{Levy-Khintchine-repreentation}
f(x)=a+bx+\int_0^\infty\bigl(1-e^{-xt}\bigr)\td\mu(t),
\end{equation}
where $a,b\ge0$ and $\mu$ is a measure on $(0,\infty)$ satisfying $\int_0^\infty\min\{1,t\}\td\mu(t)<\infty$.
In particular, the triplet $(a,b,\mu)$ determines $f$ uniquely and vice versa.
\end{prop}

The formula~\eqref{Levy-Khintchine-repreentation} is called L\'evy-Khintchine representation of $f$. The representing measure $\mu$ and the characteristic triplet $(a,b,\mu)$ from~\eqref{Levy-Khintchine-repreentation} are often respectively called L\'evy measure and L\'evy triplet of the Bernstein function $f$.

\begin{dfn}[{\cite[Definition~6.1]{Schilling-Song-Vondracek-2nd}}]
If L\'evy measure $\mu$ from~\eqref{Levy-Khintchine-repreentation} has a completely monotonic density $m(t)$ with respect to Lebesgue measure, that is, the integral representation
\begin{equation}\label{Levy-Khintchine-m(t)}
f(x)=a+bx+\int_0^\infty\bigl(1-e^{-xt}\bigr)m(t)\td t
\end{equation}
holds for $a,b\ge0$, where $m(t)$ is a completely monotonic function on $(0,\infty)$ and satisfies $\int_0^\infty\min\{1,t\}m(t)\td t<\infty$, then $f$ is said to be a complete Bernstein function on $(0,\infty)$.
\end{dfn}

We denote the collection of all complete Bernstein functions on $(0,\infty)$ by $\mathcal{CB}$.

\begin{dfn}[{\cite[Definition~8.1]{Schilling-Song-Vondracek-2nd}}]
If the function $tm(t)$ is completely monotonic on $(0,\infty)$, then $f$ is said to be a Thorin-Bernstein function on $(0,\infty)$.
\end{dfn}

We use $\mathcal{TB}$ to denote the class of all Thorin-Bernstein functions on $(0,\infty)$.

\begin{dfn}[{\cite{Geom-Mean-Deg-One-Guo.tex, simp-exp-degree-ext-gjma.tex} and~\cite[Definition~1.4]{Qi-Zhang-Li-MJM-14}}]
If $\cmdeg{t}[m(t)]=r$ for some $r\ge0$, then $f$ is said to be a complete Bernstein function of degree $r$, or say, the scalar $r$ is said to be the degree of the complete Bernstein function $f$ on $(0,\infty)$.
\end{dfn}

Similar to $\cmdeg{t}[f(t)]$, we use the notation $\cbdeg{t}[f(t)]$ to represent the degree of the complete Bernstein function $f$ on $(0,\infty)$.

\begin{dfn}[{\cite[Definition~2]{Besenyei-MIA-13}}]
Let $\mathbb{M}_n^+$ denote the space of $n\times n$ complex Hermitian positive semi-definite matrices with the usual ordering that $A\le B$ means that $B-A$ is a positive matrix. For a real function $f$ on an interval $I$, if $D$ is a diagonal matrix $\diag(\lambda_1,\lambda_2,\dotsc,\lambda_n)$, then define $f(D)=\diag(f(\lambda_1),f(\lambda_2),\dotsc,f(\lambda_n))$. If $A$ is an Hermitian matrix with eigenvalues belonging to $I$, then define $f(A)=Uf(D)U^H$, where $A=UDU^H$ and the diagonal matrix $D$ is constituted by the eigenvalues of $A$, with $U$ being a unitary matrix and $U^H$ being the conjugate transpose of $U$.
A function $f:I\to(0,\infty)$ is said to be matrix monotone of order $n$ if $A\le B$ implies $f(A)\le f(B)$, where $A,B\in\mathbb{M}_n^+$ and the eigenvalues of $A$ and $B$ are contained in the interval $I$.
If for every $n\ge1$ a function $f$ on an interval $I$ is always matrix monotone of order $n$, then $f$ is said to be operator monotone on $I$.
\end{dfn}

We use $\mathcal{OM}[I]$ to denote the set of all operator monotone functions on an interval $I$.
\par
Because~\cite[Theorem~7.3]{Schilling-Song-Vondracek-2nd}\label{Theorem-7.3-Schilling-Song-Vondracek-2nd} reads that a (non-trivial) function $f$ is a complete Bernstein function if and only if $\frac1f$ is a (non-trivial) Stieltjes function, we may define a new notion ``degree of Stieltjes function'' as follows.

\begin{dfn}\label{deg-stieltjes-dfn}
Let $f(t)$ be a Stieltjes function. If $\cbdeg{t}\bigl[\frac1{f(t)}\bigr]=r$ for some $r\ge0$, then $f(t)$ is said to be a Stieltjes function of degree $r$, or say, the scalar $r$ is said to be the degree of the Steltjes function $f$ on $(0,\infty)$.
\end{dfn}

Similar to the above mentioned $\cmdeg{t}[f(t)]$ and $\cbdeg{t}[f(t)]$, we use $\stdeg{t}[f(t)]$ to represent the degree of the Stieltjes function $f$.

\begin{rem}
Since a complete Bernstein function is of degree $\infty$ if and only if it is a linear function $a+bx$ with $a,b\ge0$, see~\cite{Qi-Zhang-Li-MJM-14}, then $\stdeg{x}[f(x)]=\infty$ if and only if $f(x)=\frac1{a+bx}$ for $(a,b)\in[0,\infty)\times[0,\infty)\setminus\{(0,0)\}$.
\end{rem}

It is stated in~\cite[Theorem~12.17]{Schilling-Song-Vondracek-2nd}\label{Theorem.12.17-SV2} that the families of complete Bernstein factions and positive operator monotone functions on $(0,\infty)$ coincide. Therefore, we may introduce a new notion ``degree of a positive operator monotone function'' as follows.

\begin{dfn}\label{deg-op-dfn}
Let $f(t)$ be a positive operator monotone function on $(0,\infty)$. If $\cbdeg{t}[f(t)]=r$ for some $r\ge0$, then $f(t)$ is said to be a positive operator monotone function of degree $r$, or say, the scalar $r$ is said to be the degree of the positive operator monotone function $f$ on $(0,\infty)$.
\end{dfn}

Similar to the above mentioned $\cmdeg{t}[f(t)]$, $\cbdeg{t}[f(t)]$, and $\stdeg{t}[f(t)]$, we use $\opdeg{t}[f(t)]$ to represent the degree of a positive operator monotone function $f$.

\begin{rem}
As mentioned above, a complete Bernstein function is of degree $\infty$ if and only if it is a linear function $a+bx$ with $a,b\ge0$, see~\cite{Qi-Zhang-Li-MJM-14}, then $\opdeg{x}[f(x)]=\infty$ if and only if $f(x)=a+bx$ for $a,b\ge0$.
\end{rem}

\section{Motivation and main results}

Now we simply summarize up the motivation of this paper.
\par
In~\cite{Ivady-JMI-09-03-02}, the double inequality
\begin{equation}\label{Ivady-JMI-09-03-02-ineq}
\frac{x^2+1}{x+1}<\Gamma(x+1)<\frac{x^2+2}{x+2}
\end{equation}
for $x\in(0,1)$ was obtained.
\par
In~\cite[Theorem~1]{refine-Ivady-gamma-PMD.tex}, the double inequality~\eqref{Ivady-JMI-09-03-02-ineq} was generalized as the following monotonicity.

\begin{thm}\label{bounds-for-gamma-thm}
The function
\begin{equation}\label{gamma-ln-ratio}
Q(x)=\frac{\ln\Gamma(x+1)}{\ln(x^2+1)-\ln(x+1)}
\end{equation}
is strictly increasing on $(0,1)$, with the limits
\begin{equation}\label{limit-ivady-1}
\lim_{x\to0^+}Q(x)=\gamma \quad \text{and}\quad \lim_{x\to1^-}Q(x)=2(1-\gamma),
\end{equation}
where $\gamma=0.57\dotsm$ denotes Euler-Mascheroni's constant.
\end{thm}

In~\cite[Section~5]{refine-Ivady-gamma-PMD.tex}, the following problem and conjectures were posed.

\begin{open}[{\cite[Section~5.1]{refine-Ivady-gamma-PMD.tex}}]\label{OP-debrecen}
What is the largest number $\tau>1$ \textup{(}or the smallest number $0\le\tau<6$, respectively\textup{)} for the function
\begin{equation}\label{gamma-ln-ratio-lambda}
f_\tau(x)=
\begin{cases}
    \dfrac{\ln\Gamma(x+1)}{\ln(x^2+\tau)-\ln(x+\tau)},& x\ne1\\
    (1+\tau)(1-\gamma), & x=1
  \end{cases}
\end{equation}
on $(0,\infty)$, where $\gamma=0.577\dotsc$ denotes Euler-Mascheroni's constant, to be strictly increasing \textup{(}or decreasing, respectively\textup{)} on $(0,1)$?
\end{open}

\begin{conj}[{\cite[Section~5.2]{refine-Ivady-gamma-PMD.tex}}]\label{gamma-ln-ratio-conj}
The function $f_1(x)$ is strictly increasing not only on $(0,1)$ but also on $(0,\infty)$.
\end{conj}

\begin{conj}[{\cite[Section~5.2]{refine-Ivady-gamma-PMD.tex}}]\label{conj-second}
For $\tau\ge0$, the function
\begin{equation}
g_\tau(x)=
\begin{cases}
    \dfrac{\ln\Gamma(x)}{\ln(x^2+\tau)-\ln(x+\tau)},& x\ne1\\
    -(1+\tau)\gamma, & x=1
  \end{cases}
\end{equation}
is strictly increasing on $(0,\infty)$.
\end{conj}

\begin{conj}[{\cite[Section~5.2]{refine-Ivady-gamma-PMD.tex}}]\label{conj-third}
For $\tau\ge0$, let
\begin{equation}\label{ln-x-frac-eq}
  h_\tau(x)=
  \begin{cases}
  \dfrac{\ln x}{\ln(x^2+\tau)-\ln(x+\tau)},&x\ne1\\
  1+\tau,&x=1
\end{cases}
\end{equation}
on $(0,\infty)$. The function $h_1$ is completely monotonic on $(0,\infty)$.
\end{conj}

Problem~\ref{OP-debrecen} was answered in~\cite{Kupan-Szasz-MIA-2014} and the answer reads that the function $f_\tau(x)$ is increasing on $(0,1)$ if and only if $0\le\tau<\frac{6\gamma}{\pi^2-12\gamma}=1.176\dotsc$ and that it is decreasing on $(0,1)$ if and only if $\tau\ge\frac{\pi^2-6\gamma}{18-12\gamma-\pi^2}=5.321\dotsc$.
Conjecture~\ref{gamma-ln-ratio-conj} was confirmed by~\cite[Theorem~1]{mon-funct-gamma-unit-ball.tex}. Consequently, by the relation
\begin{equation}
h_\tau(x)+g_\tau(x)=f_\tau(x),
\end{equation}
it follows that the function $g_1(x)$ is increasing on $(0,\infty)$, which is a partial verification to Conjecture~\ref{conj-second} for $0\le\tau\le1$.
\par
The aim of this paper is to verify Conjecture~\ref{conj-third}. To attain our aim, we need the following knowledge.
\par
It is easy to obtain that $\lim_{x\to\infty}h(x)=1$. Let
\begin{equation}\label{H-h-relation-eq}
H(x)=h(x)-1=
  \begin{cases}
  \dfrac{\ln x+\ln(1+x)-\ln(1+x^2)}{\ln(1+x^2)-\ln(1+x)},&x\ne1,\\
  1,&x=1.
\end{cases}
\end{equation}
As usual, we use $\ln x$ for the logarithmic function having base $e$ and applied to the positive argument $x>0$. Further, the principal branch of the holomorphic extension of $\ln x$ from the open half-line $(0,\infty)$ to the cut plane
$
\mathcal{A}=\mathbb{C}\setminus(-\infty,0]
$
is denoted by
$
\ln z=\ln|z|+i\arg z,
$
where $i=\sqrt{-1}\,$ and the argument of $z$ satisfies $-\pi<\arg z<\pi$. It is not difficult to see that the principal branchs of the holomorphic extensions of $h(x)$ and $H(z)$ to $\mathcal{A}$ are
\begin{equation}\label{ln-x-frac-eq-complex}
  h(z)=
  \begin{cases}
  \dfrac{\ln z}{\ln\frac{1+z^2}{1+z}},&z\ne1\\
  2,&z=1
\end{cases}
\end{equation}
and
\begin{equation}\label{ln-x-frac-eq-extension}
  H(z)=
  \begin{cases}
  \dfrac{\ln \frac{z(1+z)}{1+z^2}}{\ln\frac{1+z^2}{1+z}},&z\ne1,\\
  1,&z=1.
\end{cases}
\end{equation}
\par
By Cauchy integral formula in the theory of complex functions, we will obtain more and stronger results than Conjecture~\ref{conj-third}, which can be formulated as the following theorems.

\begin{thm}\label{third-conj-thm}
For $z\in\mathcal{A}$,
\begin{enumerate}
\item
the principal branch of the complex function $zH(z)$ has an integral representation
\begin{equation}\label{zH(z)-int-exp}
zH(z)=\int_0^\infty\frac{\rho(t)}{t+z}\td t,
\end{equation}
where
\begin{equation}
\rho(t)=
\begin{cases}
\dfrac{t}{\ln\frac{1+t^2}{1-t}},&0<t<1\\
0, & t=1\\
\dfrac{t\ln\frac{t(1+t^2)}{t-1}} {\bigl(\ln\frac{1+t^2}{t-1}\bigr)^2+\pi^2}, &1<t<1+\sqrt2\,\\
\dfrac{t\ln\frac{1+t^2}{t(t-1)}}{\bigl(\ln\frac{1+t^2}{t-1}\bigr)^2+\pi^2}, &1+\sqrt2\,\le t<\infty
\end{cases}
\end{equation}
is non-negative on $(0,\infty)$;
\item
the principal branch of the complex function $\frac1{z^2H(z)}$ has an integral representation
\begin{equation}\label{frac1z2H(z)-int}
\frac1{z^2H(z)}=\int_0^\infty\frac{\rho(t)}{\varrho(t)}\frac1{t+z}\td t,
\end{equation}
where
\begin{equation*}
\varrho(t)=
\begin{cases}
\dfrac{t\bigl\{\bigl[\ln\frac{1+t^2}{t(1-t)}\bigr]^2+\pi^2\bigr\}} {\bigl(\ln\frac{1+t^2}{1-t}\bigr)^2},&0<t<1\\
t, & t=1\\
\dfrac{t\bigl\{\bigl[\ln\frac{1+t^2}{t-1}\ln\frac{1+t^2}{t(t-1)}+2\pi^2\bigr]^2 +\bigl[\pi\ln\frac{t(1+t^2)}{t-1}\bigr]^2\bigr\}} {\bigl[\bigl(\ln\frac{1+t^2}{t-1}\bigr)^2+\pi^2\bigr]^2}, &1<t<1+\sqrt2\,\\
\dfrac{t\bigl[\ln\frac{1+t^2}{t(t-1)}\bigr]^2} {\bigl(\ln\frac{1+t^2}{t-1}\bigr)^2+\pi^2}, &1+\sqrt2\,\le t<\infty
\end{cases}
\end{equation*}
is positive on $(0,\infty)$;
\item
the principal branch of the complex function $h(z)$ has an integral representation
\begin{equation}\label{h(z)-int-representation}
H(z)=h(z)-1=\int_0^\infty\biggl[\int_0^\infty \frac{\rho(u)}{u}\bigl(1-e^{-t u}\bigr)\td u\biggr]e^{-tz}\td t;
\end{equation}
\item
the principal branch of the complex function $\frac1{H(z)}$ has an integral representation
\begin{equation}\label{frac1H(z)-int}
\frac1{H(z)}=z\int_0^\infty \frac{\rho(t)}{\varrho(t)}\td t -\int_0^\infty\bigl(1-e^{-zt}\bigr) \biggl[\int_0^\infty \frac{u^2\rho(u)}{\varrho(u)}e^{-ut}\td u\biggr]\td t.
\end{equation}
\end{enumerate}
\end{thm}

\begin{thm}\label{third-conj-thm-conseq}
For $x\in(0,\infty)$ and $z\in\mathcal{A}$,
\begin{enumerate}
\item
$h(x),H(x)\in\mathcal{C}[(0,\infty)]$, with the same integral representation~\eqref{h(z)-int-representation} and
\begin{equation}\label{H(z)-stieljes-int}
H(z)=\frac1z\int_0^\infty\frac{\rho(t)}t\td t-\int_0^\infty\frac{\rho(t)}{t(t+z)}\td t
\end{equation}
and of degree
\begin{equation}\label{H(x)-deg=1-eq}
\cmdeg{x}[h(x)]=\cmdeg{x}[H(x)]=1;
\end{equation}
\item
$\frac1{H(x)}\in\mathcal{B}[(0,\infty)]$ and $H(x)\in\mathcal{L}[(0,\infty)]$;
\item
$xH(x)\in\mathcal{S}$ and $xH(x)\in\mathcal{L}[(0,\infty)]$, with the integral representation~\eqref{zH(z)-int-exp} and of degree
\begin{equation}\label{xH(x)-st-deg-eq}
\stdeg{x}[xH(x)]=0;
\end{equation}
\item
$\frac1{xH(x)}\in\mathcal{CB}$ and $\frac1{xH(x)}\in\mathcal{OM}[(0,\infty)]$, with L\'evy-Khintchine representation
\begin{equation}\label{zH(z)-CB-int-eq}
\frac1{zH(z)}=\int_0^\infty\biggl[\int_0^\infty\frac{u\rho(u)}{\varrho(u)} e^{-tu}\td u\biggl] \bigl(1-e^{-zt}\bigr)\td t
\end{equation}
and of degree
\begin{equation}\label{xH(x)-cb-deg-eq}
\cbdeg{x}\biggl[\frac1{xH(x)}\biggr]=0 \quad \text{and}\quad \opdeg{x}\biggl[\frac1{xH(x)}\biggr]=0;
\end{equation}
\item
$\frac1{x^2H(x)}\in\mathcal{S}$ and $\frac1{x^2H(x)}\in\mathcal{L}[(0,\infty)]$, with the integral representation~\eqref{frac1z2H(z)-int} and of degree
\begin{equation}\label{x^2H(x)-st-deg-eq}
\stdeg{x}\biggl[\frac1{x^2H(x)}\biggr]=0;
\end{equation}
\item
$x^2H(x)\in\mathcal{CB}$ and $x^2H(x)\in\mathcal{OM}[(0,\infty)]$, with L\'evy-Khintchine representation
\begin{equation}\label{z^2H(z)-int-eq}
z^2H(z)=\int_0^\infty\biggl[\int_0^\infty{u\rho(u)}e^{-tu}\td u\biggl]\bigl(1-e^{-zt}\bigr)\td t
\end{equation}
and of degree
\begin{equation}\label{x^2H(x)-cb-deg-eq}
\cbdeg{x}\bigl[x^2H(x)\bigr]=0\quad \text{and}\quad \opdeg{x}\bigl[x^2H(x)\bigr]=0.
\end{equation}
\end{enumerate}
\end{thm}

\section{Remarks}

Before proving Theorems~\ref{third-conj-thm} and~\ref{third-conj-thm-conseq}, we supply some remarks on them.

\begin{rem}
It is listed in~\cite[Proposition~7.1]{Schilling-Song-Vondracek-2nd} and~\cite[p.~96]{Schilling-Song-Vondracek-2nd} that for $f(x)>0$,
\begin{enumerate}
\item
$f\in\mathcal{CB}$ if and only if $\frac{x}{f(x)}\in\mathcal{CB}$,
\item
$f(x)\in\mathcal{CB}$ if and only if $\frac1{f(1/x)}\in\mathcal{CB}$,
\item
$f(x)\in\mathcal{CB}$ if and only if $xf\bigl(\frac1x\bigr)\in\mathcal{CB}$,
\item
$f(x)\in\mathcal{S}$ if and only if $\frac1{f(1/x)}\in\mathcal{S}$,
\item
$f(x)\in\mathcal{S}$ if and only if $\frac1{xf(x)}\in\mathcal{S}$,
\item
$f(x)\in\mathcal{S}$ if and only if $\frac{f(x)}{\epsilon f(x)+1}\in\mathcal{S}$ for all $\epsilon>0$.
\end{enumerate}
From these properties and the fact that $H\bigl(\frac1x\bigr)=\frac1{H(x)}$, it follows that
\begin{enumerate}
\item
$\frac1{xH(x)}\in\mathcal{CB}$ if and only if ${x^2}{H(x)}\in\mathcal{CB}$,
\item
$xH(x)\in\mathcal{S}$ if and only if $\frac1{x^2H(x)}\in\mathcal{S}$,
\item
$xH(x)\in\mathcal{S}$ if and only if $\frac{xH(x)}{\epsilon xH(x)+1}\in\mathcal{S}$ for all $\epsilon>0$.
\end{enumerate}
\end{rem}

\begin{rem}
Corollary~7.4 in~\cite{Schilling-Song-Vondracek-2nd} states that $0<g\in\mathcal{S}$ if and only if $g(0^+)$ exists in $[0,\infty]$ and $g$ extends analytically to $\mathcal{A}$ such that $\Im z\cdot\Im g(z)\le0$. This means that
\begin{equation}
\Im z\cdot\Im [zH(z)]\le0, \quad \Im z\cdot\Im \frac1{x^2H(x)}\le0, \quad \text{and}\quad \Im z\cdot\Im \frac{zH(z)}{\epsilon zH(z)+1}\le0
\end{equation}
for all $\epsilon>0$ and $z\in\mathcal{A}$. Geometrically speaking, the three Stieltjes functions $zH(z)$, $\frac1{x^2H(x)}$, and $\frac{zH(z)}{\epsilon zH(z)+1}$ for all $\epsilon>0$ map the upper half-plane to the lower half-plane and vice versa.
\end{rem}

\section{Lemmas}
In order to prove our main results, the following lemmas are necessary.

\begin{lem}[Cauchy integral formula~{\cite[p.~113]{Gamelin-book-2001}}]\label{cauchy-formula}
Let $D$ be a bounded domain with piecewise smooth boundary. If $f(z)$ is analytic on $D$, and $f(z)$ extends smoothly to the boundary of $D$, then
\begin{equation}
f(z)=\frac1{2\pi i}\oint_{\partial D}\frac{f(w)}{w-z}\td w,\quad z\in D.
\end{equation}
\end{lem}

\begin{lem}\label{lem-zk}
For $z=re^{\theta i}\in\mathcal{A}$, the complex function $zH(z)$ uniformly tends to $0$ as $r\to\infty$.
\end{lem}

\begin{proof}
By standard argument, we have
\begin{align*}
\lim_{r\to\infty}|zH(z)|^2&=\lim_{r\to0^+}\biggl|\frac1zH\biggl(\frac1z\biggr)\biggr|^2
=\lim_{r\to0^+}\biggl|\frac1{zH(z)}\biggr|^2 \\
&=\lim_{r\to0^+}\frac {\bigl|\ln\bigl|\frac{1+z^2}{1+z}\bigr| +i\arg\frac{1+z^2}{1+z}\bigr|^2}{|z|^2\bigl|\ln\bigl|\frac{z(1+z)}{1+z^2}\bigr| +i\arg\frac{z(1+z)}{1+z^2}\bigr|^2}\\
&=\lim_{r\to0^+}\frac {\bigl|\ln\bigl|\frac{1+z^2}{1+z}\bigr|\bigr|^2+\bigl|\arg\frac{1+z^2}{1+z}\bigr|^2} {|z|^2\bigl[\bigl|\ln\bigl|\frac{z(1+z)}{1+z^2}\bigr|\bigr|^2 +\bigl|\arg\frac{z(1+z)}{1+z^2}\bigr|^2\bigr]}\\
&=\lim_{r\to0^+}\frac{\bigl[\frac1{2r}\ln\frac{1+2r\cos\theta+r^2}{1+2r^2\cos(2\theta)+r^4}\bigr]^2 +\bigl[\frac1r\arctan\frac{r(1-2 r \cos\theta-r^2)\sin\theta}{1+r(1+r^2) \cos\theta+r^2 \cos (2 \theta)}\bigr]^2} {\bigl[\ln r+\frac12\ln\frac {1+2r\cos\theta+r^2} {1+2r^2\cos(2\theta)+r^4}\bigr]^2 +\bigl|\arg\frac{z(1+z)}{1+z^2}\bigr|^2},
\end{align*}
where
\begin{align*}
\lim_{r\to0^+}\biggl[\frac1{2r}\ln\frac{1+2r\cos\theta+r^2}{1+2r^2\cos(2\theta)+r^4}\biggr] &=\cos\theta,\\
\lim_{r\to0^+}\biggl[\frac1r\arctan\frac{r(1-2 r \cos\theta-r^2)\sin\theta}{1+r(1+r^2) \cos\theta+r^2 \cos (2 \theta)}\biggr]&=\sin\theta,\\
\lim_{r\to0^+}\ln\frac {1+2r\cos\theta+r^2} {1+2r^2\cos(2\theta)+r^4}&=0,
\end{align*}
and $0\le\bigl|\arg\frac{z(1+z)}{1+z^2}\bigr|\le\pi$. Further considering $\lim_{r\to0^+}\ln r=-\infty$, Lemma~\ref{lem-zk} is thus proved.
\end{proof}

\begin{lem}\label{lem-zk-2}
For $z=re^{\theta i}\in\mathcal{A}$ and $\theta\in\bigl[-\frac\pi2,\frac\pi2\bigr]$, the complex function $z^2H(z)$ uniformly tends to $0$ as $r\to0^+$.
\end{lem}

\begin{proof}
By the same argument as in the proof of Lemma~\ref{lem-zk}, we have
\begin{multline*}
\lim_{r\to0^+}\bigl|z^2H(z)\bigr|^2=\lim_{r\to0^+}\Biggl|z^2\frac{\ln \bigl|\frac{z(1+z)}{1+z^2}\bigr| +i\arg\frac{z(1+z)}{1+z^2}}{\ln\bigl|\frac{1+z^2}{1+z}\bigr|+i\arg\frac{1+z^2}{1+z}}\Biggr|^2\\
=\lim_{r\to0^+}\frac{\bigl[r\ln r+\frac{r}2\ln\frac {1+2r\cos\theta+r^2} {1+2r^2\cos(2\theta)+r^4}\bigr]^2 +r^2\bigl|\arg\frac{z(1+z)}{1+z^2}\bigr|^2} {\bigl[\frac1{2r}\ln\frac{1+2r\cos\theta+r^2}{1+2r^2\cos(2\theta)+r^4}\bigr]^2 +\bigl[\frac1r\arctan\frac{r(1-2 r \cos\theta-r^2)\sin\theta}{1+r(1+r^2) \cos\theta+r^2 \cos (2 \theta)}\bigr]^2}=0.
\end{multline*}
The proof of Lemma~\ref{lem-zk-2} is complete.
\end{proof}

\begin{lem}\label{lem-zk-3}
For $z=re^{\theta i}\in\mathcal{A}$, the limits
\begin{equation}
\lim_{z\to0}\frac1{zH(z)}=0
\end{equation}
for $\theta\in\bigl[-\frac\pi2,\frac\pi2\bigr]$ and
\begin{equation}
\lim_{z\to\infty}\frac1{z^2H(z)}=0
\end{equation}
for $\theta\in(-\pi,\pi)$ converge uniformly.
\end{lem}

\begin{proof}
This follows from similar arguments as in proofs of Lemmas~\ref{lem-zk} and~\ref{lem-zk-2}.
\end{proof}

\begin{lem}\label{lem-h-Im}
For $z\in\mathcal{A}$, the real and imaginary parts of the principal branch of the complex function $zH(z)$ satisfy
\begin{equation*}
\lim_{\varepsilon\to0^+}\Re[(-t+\varepsilon i)H(-t+\varepsilon i)]
=
\begin{cases}
\dfrac{t\ln\frac{1+t^2}{t(1-t)}}{\ln\frac{1+t^2}{1-t}},&0<t<1;\\
t, & t=1;\\
\dfrac{t\bigl[\ln\frac{1+t^2}{t-1}\ln\frac{1+t^2}{t(t-1)}+2\pi^2\bigr]} {\bigl(\ln\frac{1+t^2}{t-1}\bigr)^2+\pi^2}, &1<t<1+\sqrt2\,;\\
\dfrac{t\ln\frac{1+t^2}{t-1}\ln\frac{1+t^2}{t(t-1)}}{\bigl(\ln\frac{1+t^2}{t-1}\bigr)^2+\pi^2}, &1+\sqrt2\,\le t<\infty
\end{cases}
\end{equation*}
and
\begin{equation}\label{lem-zH(z)-Im-eq}
\lim_{\varepsilon\to0^+}\Im[(-t+\varepsilon i)H(-t+\varepsilon i)]
=
\begin{cases}
-\dfrac{\pi t}{\ln\frac{1+t^2}{1-t}},&0<t<1;\\
0, & t=1;\\
-\dfrac{\pi t\ln\frac{t(1+t^2)}{t-1}} {\bigl(\ln\frac{1+t^2}{t-1}\bigr)^2+\pi^2}, &1<t<1+\sqrt2\,;\\
-\dfrac{\pi t\ln\frac{1+t^2}{t(t-1)}}{\bigl(\ln\frac{1+t^2}{t-1}\bigr)^2+\pi^2}, &1+\sqrt2\,\le t<\infty.
\end{cases}
\end{equation}
\end{lem}

\begin{proof}
For $t\in(0,\infty)$ and $\varepsilon>0$, when $z=-t+\varepsilon i$, we have
\begin{align*}
\ln\frac{1+z^2}{1+z}&=\ln\frac{1+(-t+\varepsilon i)^2}{1-t+\varepsilon i}\\
&=\ln\frac{\bigl[1+(-t+\varepsilon i)^2\bigr](1-t-\varepsilon i)}{(1-t)^2+\varepsilon^2}\\
&=\ln\biggl|\frac{(1-t)\bigl(t^2+1\bigr)-\varepsilon^2 (t+1)+i \varepsilon\bigl(\varepsilon^2+t^2-2 t-1\bigr)}{(1-t)^2+\varepsilon^2}\biggr|\\
&\quad+i\arg\frac{(1-t)\bigl(t^2+1\bigr)-\varepsilon^2 (t+1)+i \varepsilon\bigl(\varepsilon^2+t^2-2 t-1\bigr)}{(1-t)^2+\varepsilon^2}\\
&\to
\begin{cases}
\ln\dfrac{1+t^2}{1-t},&0<t<1\\
\infty-\dfrac\pi2i, & t=1\\
\ln\dfrac{1+t^2}{t-1}-\pi i,&1<t<1+\sqrt2\,\\
\ln\dfrac{1+t^2}{t-1}+\pi i,&1+\sqrt2\,\le t<\infty
\end{cases}
\end{align*}
as $\varepsilon\to0^+$. So, it follows that
\begin{align*}
h(-t+\varepsilon i)
&=\frac{\ln(-t+\varepsilon i)}{\ln\frac{1+(-t+\varepsilon i)^2}{1-t+\varepsilon i}}\\
&\to
\begin{cases}
\dfrac{\ln t+\pi i}{\ln\frac{1+t^2}{1-t}},&0<t<1\\
0, & t=1\\
\dfrac{\ln t+\pi i}{\ln\frac{1+t^2}{t-1}-\pi i},&1<t<1+\sqrt2\,\\
\dfrac{\ln t+\pi i}{\ln\frac{1+t^2}{t-1}+\pi i},&1+\sqrt2\,\le t<\infty
\end{cases}\\
&=
\begin{cases}
\dfrac{\ln t}{\ln\frac{1+t^2}{1-t}}+\dfrac{\pi}{\ln\frac{1+t^2}{1-t}} i,&0<t<1\\
0, & t=1\\
\dfrac{\ln t\ln\frac{1+t^2}{t-1}-\pi^2}{\bigl(\ln\frac{1+t^2}{t-1}\bigr)^2+\pi^2} +\dfrac{\pi\bigl(\ln\frac{1+t^2}{t-1}+\ln t\bigr)} {\bigl(\ln\frac{1+t^2}{t-1}\bigr)^2+\pi^2}i,&1<t<1+\sqrt2\,\\
\dfrac{\ln t\ln\frac{1+t^2}{t-1}+\pi^2}{\bigl(\ln\frac{1+t^2}{t-1}\bigr)^2+\pi^2} +\dfrac{\pi\bigl(\ln\frac{1+t^2}{t-1}-\ln t\bigr)}{\bigl(\ln\frac{1+t^2}{t-1}\bigr)^2+\pi^2}i, &1+\sqrt2\,\le t<\infty
\end{cases}
\end{align*}
as $\varepsilon\to0^+$. As a result,
\begin{equation}\label{lem-h-Re-eq}
\lim_{\varepsilon\to0^+}\Re h(-t+\varepsilon i)
=
\begin{cases}
\dfrac{\ln t}{\ln\frac{1+t^2}{1-t}},&0<t<1;\\
0, & t=1;\\
\dfrac{\ln t\ln\frac{1+t^2}{t-1}-\pi^2}{\bigl(\ln\frac{1+t^2}{t-1}\bigr)^2+\pi^2}, &1<t<1+\sqrt2\,;\\
\dfrac{\ln t\ln\frac{1+t^2}{t-1}+\pi^2}{\bigl(\ln\frac{1+t^2}{t-1}\bigr)^2+\pi^2}, &1+\sqrt2\,\le t<\infty
\end{cases}
\end{equation}
and
\begin{equation}\label{lem-h-Im-eq}
\lim_{\varepsilon\to0^+}\Im h(-t+\varepsilon i)
=
\begin{cases}
\dfrac{\pi}{\ln\frac{1+t^2}{1-t}},&0<t<1;\\
0, & t=1;\\
\dfrac{\pi\bigl(\ln\frac{1+t^2}{t-1}+\ln t\bigr)} {\bigl(\ln\frac{1+t^2}{t-1}\bigr)^2+\pi^2}, &1<t<1+\sqrt2\,;\\
\dfrac{\pi\bigl(\ln\frac{1+t^2}{t-1}-\ln t\bigr)}{\bigl(\ln\frac{1+t^2}{t-1}\bigr)^2+\pi^2}, &1+\sqrt2\,\le t<\infty.
\end{cases}
\end{equation}
\par
From the relation~\eqref{H-h-relation-eq} between $h(x)$ and $H(x)$ and the property of complex numbers, it follows that
\begin{equation*}
\Re[zH(z)]=\Re[zh(z)-z]=\Re z\Re h(z)-\Im z\Im h(z)-\Re z
\end{equation*}
and
\begin{equation*}
\Im[zH(z)]=\Im[zh(z)-z]=\Im z\Re h(z)+\Re z\Im h(z)-\Im z.
\end{equation*}
Accordingly, we obtain
\begin{equation*}
\lim_{\varepsilon\to0^+}\Re[(-t+\varepsilon i)H(-t+\varepsilon i)]=-t\lim_{\varepsilon\to0^+}\Re h(z)+t =t\Bigl[1-\lim_{\varepsilon\to0^+}\Re h(z)\Bigr]
\end{equation*}
and
\begin{equation*}
\lim_{\varepsilon\to0^+}\Im[(-t+\varepsilon i)H(-t+\varepsilon i)]=-t\lim_{\varepsilon\to0^+}\Im h(z).
\end{equation*}
Combining these with the limits~\eqref{lem-h-Re-eq} and~\eqref{lem-h-Im-eq} and simplifying yield the required limits. Lemma~\ref{lem-h-Im} is thus proved.
\end{proof}

\begin{lem}[{Convolution theorem for Laplace transforms~\cite[pp.~91\nobreakdash--92]{Widder-Laplace-Transform-41}}]\label{convlotion}
Let $f_i(t)$ for $i=1,2$ be piecewise continuous in arbitrary finite intervals included in $(0,\infty)$. If there exist some constants $M_i>0$ and $c_i\ge0$ such that $|f_i(t)|\le M_ie^{c_it}$ for $i=1,2$, then
\begin{equation}
\int_0^\infty\bigg[\int_0^tf_1(u)f_2(t-u)\td u\bigg]e^{-st}\td t =\int_0^\infty
f_1(u)e^{-su}\td u\int_0^\infty f_2(v)e^{-sv}\td v.
\end{equation}
\end{lem}

\section{Proofs of theorems}

We now start out to to prove our main results stated in Theorems~\ref{third-conj-thm} and~\ref{third-conj-thm-conseq}.

\begin{proof}[Proof of Theorem~\ref{third-conj-thm}]
For any fixed point $z_0=x_0+iy_0\in\mathbb{C}\setminus(-\infty,0]$, choose $\varepsilon$ and $r$ such that \begin{equation*}
\begin{cases}
0<\varepsilon<|y_0|\le|z_0|<r, & y_0\ne0,\\
0<\varepsilon<x_0=|z_0|<r, & y_0=0,\\
\end{cases}
\end{equation*}
and consider the positively oriented contour $C(\varepsilon,r)$ in $\mathbb{C}\setminus(-\infty,0]$ consisting of the half circle $z=\varepsilon e^{i\theta}$ for $\theta\in\bigl[-\frac\pi2,\frac\pi2\bigr]$ and the half lines $z=x\pm i\varepsilon$ for $x\le0$ until they cut the circle $|z|=r$, which close the contour at the points $-r(\varepsilon)\pm i\varepsilon$, where $0<r(\varepsilon)\to r$ as $\varepsilon\to0$. See Figure~\ref{note-on-li-chen-conj.eps}.
\begin{figure}[hbtp]
 \includegraphics[width=0.7\textwidth]{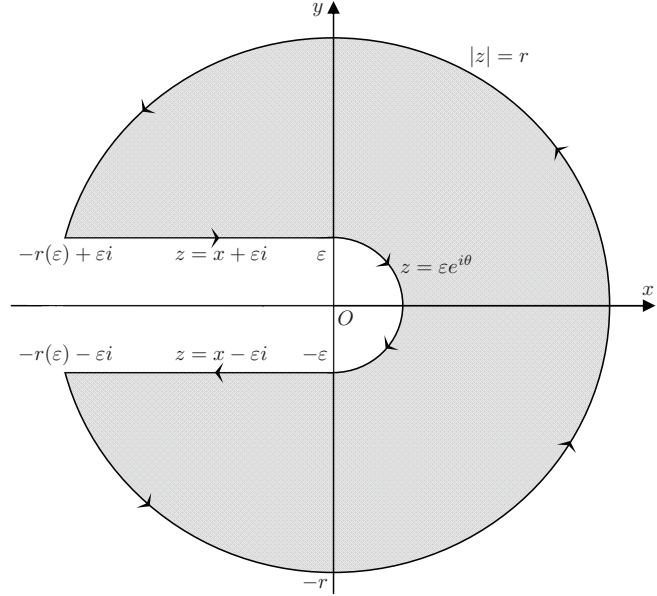}\\
 \caption{The contour $C(\varepsilon,r)$}\label{note-on-li-chen-conj.eps}
\end{figure}
\par
For our own convenience, in what follows, let us denote $zH(z)$ by $G(z)$.
\par
Applying Cauchy integral formula in Lemma~\ref{cauchy-formula} to the function $G(z)$ on the domain enclosed by the contour $C(\varepsilon,r)$ gives
\begin{multline}\label{g-Cauchy-Apply}
G(z_0)=\frac1{2\pi i}\biggl[\int_{\pi/2}^{-\pi/2}\frac{\varepsilon i e^{i\theta}G\bigl(\varepsilon e^{i\theta}\bigr)}{\varepsilon e^{i\theta}-z_0}\td\theta +\int_{\arg[-r(\varepsilon)-\varepsilon i]}^{\arg[-r(\varepsilon)+\varepsilon i]}\frac{ir e^{i\theta}G\bigl(re^{i\theta}\bigr)}{re^{i\theta}-z_0}\td\theta\\
+\int_{-r(\varepsilon)}^0 \frac{G(t+\varepsilon i)}{t+\varepsilon i-z_0}\td t +\int_0^{-r(\varepsilon)}\frac{G(t-\varepsilon i)}{t-\varepsilon i-z_0}\td t\biggr].
\end{multline}
\par
By virtue of Lemma~\ref{lem-zk}, we deduce that
\begin{equation}\label{big-circle-int=0}
\lim_{\substack{\varepsilon\to0^+\\r\to\infty}} \int_{\arg[-r(\varepsilon)-\varepsilon i]}^{\arg[-r(\varepsilon)+\varepsilon i]}\frac{ir e^{i\theta}G\bigl(re^{i\theta}\bigr)}{re^{i\theta}-z_0}\td\theta
=\lim_{r\to\infty}\int_{-\pi}^{\pi}\frac{ir e^{i\theta}G\bigl(re^{i\theta}\bigr)}{re^{i\theta}-z_0}\td\theta
=0.
\end{equation}
From Lemma~\ref{lem-zk-2}, we see that $\lim_{z\to0^+}[zG(z)]=0$. Hence, we have
\begin{equation}\label{z_0f(z_0)=0}
\lim_{\varepsilon\to0^+}\int_{\pi/2}^{-\pi/2}\frac{\varepsilon i e^{i\theta}G\bigl(\varepsilon e^{i\theta}\bigr)}{\varepsilon e^{i\theta}-z_0}\td\theta=0.
\end{equation}
Utilizing the readily verified formula $G(\overline{z})=\overline{G(z)}$ and the limit~\eqref{lem-zH(z)-Im-eq} in Lemma~\ref{lem-h-Im} results in
\begin{align*}
&\quad\int_{-r(\varepsilon)}^0 \frac{G(t+\varepsilon i)}{t+\varepsilon i-z_0}\td t
+\int_0^{-r(\varepsilon)}\frac{G(t-\varepsilon i)}{t-\varepsilon i-z_0}\td t \\
&=\int_{-r(\varepsilon)}^0 \biggl[\frac{G(t+\varepsilon i)}{t+\varepsilon i-z_0}-\frac{G(t-\varepsilon i)}{t-\varepsilon i-z_0}\biggr]\td t\\
&=\int_{-r(\varepsilon)}^0\frac{(t-\varepsilon i-z_0)G(t+\varepsilon i) -(t+\varepsilon i-z_0)G(t-\varepsilon i)} {(t+\varepsilon i-z_0)(t-\varepsilon i-z_0)}\td t\\
&=\int_{-r(\varepsilon)}^0\frac{(t-z_0)[G(t+\varepsilon i)-G(t-\varepsilon i)] -\varepsilon i[G(t-\varepsilon i)+G(t+\varepsilon i)]} {(t+\varepsilon i-z_0)(t-\varepsilon i-z_0)}\td t\\
&=2i\int_{-r(\varepsilon)}^0\frac{(t-z_0)\Im G(t+\varepsilon i) -\varepsilon\Re G(t+\varepsilon i)} {(t+\varepsilon i-z_0)(t-\varepsilon i-z_0)}\td t\\
&\to2i\int_{-r}^0\frac{\lim_{\varepsilon\to0^+}\Im G(t+\varepsilon i)}{t-z_0}\td t\\
&=-2i\int^r_0\frac{\lim_{\varepsilon\to0^+}\Im G(-t+\varepsilon i)}{t+z_0}\td t\\
&\to-2i\int^\infty_0\frac{\lim_{\varepsilon\to0^+}\Im G(-t+\varepsilon i)}{t+z_0}\td t\\
&=2\pi i\int_{0}^{\infty}\frac{\rho(t)}{t+z_0}\td t
\end{align*}
as $\varepsilon\to0^+$ and $r\to\infty$.
Substituting~\eqref{big-circle-int=0}, \eqref{z_0f(z_0)=0}, and the above one into~\eqref{g-Cauchy-Apply} and simplifying generate the integral representation~\eqref{zH(z)-int-exp}.
\par
It is standard to show that the functions $\frac{t(1+t^2)}{t-1}$ on $\bigl(1,1+\sqrt2\,\bigr)$ and $\frac{1+t^2}{t(t-1)}$ on $\bigl[1+\sqrt2\,,\infty\bigr)$ are bigger than $1$. This implies that the function $\rho(t)$ is non-negative on $(0,\infty)$.
\par
By standard argument and Lemma~\ref{lem-h-Im}, we may gain
\begin{align*}
&\quad\lim_{\varepsilon \to0^+}\Im\frac1{(-t+\varepsilon i)^2H(-t+\varepsilon i)}\\
&=\lim_{\varepsilon \to0^+}\Im\frac{-(\varepsilon \Im G(-t+\varepsilon i) + t\Re G(-t+\varepsilon i)) +i(t\Im G(-t+\varepsilon i)-\varepsilon \Re G(-t+\varepsilon i))} {(\varepsilon^2+t^2)|G(-t+\varepsilon i)|^2}\\
&=\lim_{\varepsilon \to0^+}\frac{t\Im G(-t+\varepsilon i)-\varepsilon \Re G(-t+\varepsilon i)} {(\varepsilon^2+t^2)|G(-t+\varepsilon i)|^2}\\
&=\frac{\lim_{\varepsilon \to0^+}\Im G(-t+\varepsilon i)} {t\lim_{\varepsilon \to0^+}|G(-t+\varepsilon i)|^2}
\end{align*}
and
\begin{multline*}
\lim_{\varepsilon \to0^+}|G(-t+\varepsilon i)|^2=\lim_{\varepsilon \to0^+}[\Re G(-t+\varepsilon i)]^2 +\lim_{\varepsilon \to0^+}[\Im G(-t+\varepsilon i)]^2\\
=
\begin{cases}
\dfrac{t^2\bigl\{\bigl[\ln\frac{1+t^2}{t(1-t)}\bigr]^2+\pi^2\bigr\}} {\bigl(\ln\frac{1+t^2}{1-t}\bigr)^2},&0<t<1;\\
t^2, & t=1;\\
\dfrac{t^2\bigl\{\bigl[\ln\frac{1+t^2}{t-1}\ln\frac{1+t^2}{t(t-1)}+2\pi^2\bigr]^2 +\pi^2\bigl[\ln\frac{t(1+t^2)}{t-1}\bigr]^2\bigr\}} {\bigl[\bigl(\ln\frac{1+t^2}{t-1}\bigr)^2+\pi^2\bigr]^2}, &1<t<1+\sqrt2\,;\\
\dfrac{t^2\bigl[\ln\frac{1+t^2}{t(t-1)}\bigr]^2} {\bigl(\ln\frac{1+t^2}{t-1}\bigr)^2+\pi^2}, &1+\sqrt2\,\le t<\infty.
\end{cases}
\end{multline*}
By Lemma~\ref{lem-zk-3} and the same argument as in the proof of the integral representation~\eqref{zH(z)-int-exp}, it follows that
\begin{align*}
\frac1{z^2H(z)}&=-\frac1\pi\int^\infty_0\frac{1}{t+z} \lim_{\varepsilon \to0^+}\Im\frac1{(-t+\varepsilon i)^2H(-t+\varepsilon i)}\td t\\
&=-\frac1\pi\int^\infty_0\frac{1}{t+z} \frac{\lim_{\varepsilon \to0^+}\Im G(-t+\varepsilon i)} {t\lim_{\varepsilon \to0^+}|G(-t+\varepsilon i)|^2}\td t\\
&=\int^\infty_0\frac{1}{t+z} \frac{\rho(t)} {\varrho(t)}\td t.
\end{align*}
The integral representation~\eqref{frac1z2H(z)-int} is proved.
\par
By Lemma~\ref{convlotion} and the integral representation~\eqref{zH(z)-int-exp}, it is not difficult to see that
\begin{align*}
h(z)&=1+\frac{zH(z)}{z}\\
&=1+\frac1z\int_0^\infty\frac{\rho(t)}{t+z}\td t\\
&=1+\int_0^\infty e^{-zt}\td t\int_0^\infty\int_0^\infty\rho(t)e^{-u(t+z)}\td u\td t\\
&=1+\int_0^\infty e^{-zt}\td t\int_0^\infty\biggl[\int_0^\infty\rho(u)e^{-tu}\td u\biggr]e^{-tz}\td t\\
&=1+\int_0^\infty\biggl[\int_0^t\int_0^\infty\rho(u)e^{-vu}\td u\td v\biggr]e^{-tz}\td t\\
&=1+\int_0^\infty\biggl[\int_0^\infty\int_0^t\rho(u)e^{-vu}\td v\td u\biggr]e^{-tz}\td t\\
&=1+\int_0^\infty\biggl[\int_0^\infty\frac{1-e^{-t u}}{u}\rho(u)\td u\biggr]e^{-tz}\td t.
\end{align*}
The integral representation~\eqref{h(z)-int-representation} is proved.
\par
The integral representation~\eqref{frac1z2H(z)-int} may be rearranged as
\begin{align*}
\frac1{H(z)}&=\int_0^\infty\frac{\rho(t)}{\varrho(t)}\frac{z^2}{t+z}\td t\\
&=\int_0^\infty\frac{\rho(t)}{\varrho(t)}\biggl(\frac{t^2}{t+z}+z-t\biggr)\td t\\
&=\int_0^\infty\int_0^\infty \frac{t^2\rho(t)}{\varrho(t)}e^{-(t+z)u}\td u\td t+ z\int_0^\infty \frac{\rho(t)}{\varrho(t)}\td t-\int_0^\infty\frac{t\rho(t)}{\varrho(t)}\td t\\
&=\int_0^\infty\biggl[\int_0^\infty \frac{u^2\rho(u)}{\varrho(u)}e^{-ut}\td u\biggr]e^{-zt}\td t + z\int_0^\infty \frac{\rho(t)}{\varrho(t)}\td t-\int_0^\infty\frac{t\rho(t)}{\varrho(t)}\td t.
\end{align*}
Further considering $\lim_{x\to0^+}\frac1{H(x)}=0$ and simplifying leads to the integral representation~\eqref{frac1H(z)-int}.
The proof of Theorem~\ref{third-conj-thm} is complete.
\end{proof}

\begin{proof}[Proof of Theorem~\ref{third-conj-thm-conseq}]
The properties $h(x),H(x)\in\mathcal{C}[(0,\infty)]$ follow immediately from the integral representation~\eqref{h(z)-int-representation}.
\par
The integral representation~\eqref{H(z)-stieljes-int} is a rearrangement of~\eqref{zH(z)-int-exp}.
\par
The integral representation~\eqref{zH(z)-int-exp} shows that $\cmdeg{x}H(x)\ge1$. On the other hand, if $x^\alpha H(x)\in\mathcal{C}[(0,\infty)]$, then its first derivative is non-negative, that is,
\begin{equation*}
\alpha\le-\frac{xH'(x)}{H(x)}=\frac{x \bigl(x^2+2 x-1\bigr) \ln \frac{x(x+1)}{x^2+1}+\bigl(x^2-2 x-1\bigr) \ln\frac{x^2+1}{x+1}}{(x+1) \bigl(x^2+1\bigr) \ln\frac{x (x+1)}{x^2+1} \ln\frac{x^2+1}{x+1}}\to1
\end{equation*}
as $x\to0^+$. This implies that $\cmdeg{x}H(x)\le1$. Hence, the degree in~\eqref{H(x)-deg=1-eq} holds.
\par
The property $\frac1{H(x)}\in\mathcal{B}[(0,\infty)]$ follows readily from~\eqref{frac1H(z)-int}. The property $H(x)\in\mathcal{L}[(0,\infty)]$ may be conclude from~\cite[pp.~161\nobreakdash--162, Theorem~3]{Chen-Qi-Srivastava-09.tex} which reads that if $f\in\mathcal{B}[I]$ then $\frac1f\in\mathcal{L}[I]$ for any interval $I\subseteq\mathbb{R}$.
\par
Comparing~\eqref{dfn-stieltjes} with~\eqref{zH(z)-int-exp} acquires $xH(x)\in\mathcal{S}$. By~\cite[Theorem~7.3]{Schilling-Song-Vondracek-2nd}, it follows that $\frac1{xH(x)}\in\mathcal{CB}$. By the integral representation~\eqref{frac1z2H(z)-int}, it follows that
\begin{align*}
\frac1{zH(z)}&=\int_0^\infty\frac{\rho(t)}{\varrho(t)}\frac{z}{t+z}\td t \\ &=\int_0^\infty\frac{t\rho(t)}{\varrho(t)}\biggl(\frac1t-\frac{1}{t+z}\biggr)\td t\\
&=\int_0^\infty\frac{t\rho(t)}{\varrho(t)} \biggl[\int_0^\infty e^{-tu}\td u-\int_0^\infty e^{-(t+z)u}\td u\biggr]\td t\\
&=\int_0^\infty\biggl[\int_0^\infty\frac{t\rho(t)}{\varrho(t)} e^{-tu}\td t\biggl]\bigl(1-e^{-zu}\bigr)\td u.
\end{align*}
L\'evy-Khintchine representation~\eqref{zH(z)-CB-int-eq} follows.
\par
It is obvious that the density of L\'evy-Khintchine representation~\eqref{zH(z)-CB-int-eq} is
\begin{equation*}
m(t)=\int_0^\infty \frac{u\rho(u)}{\varrho(u)}e^{-ut}\td u\in\mathcal{C}[(0,\infty)].
\end{equation*}
If $t^\alpha m(t)\in\mathcal{C}[(0,\infty)]$, then, as discussed above,
\begin{equation*}
\alpha\le-\frac{tm'(t)}{m(t)}=\frac{t\int_0^\infty \frac{u^2\rho(u)}{\varrho(u)}e^{-ut}\td u} {\int_0^\infty \frac{u\rho(u)}{\varrho(u)}e^{-ut}\td u}\to0
\end{equation*}
as $t\to0^+$. This implies the degrees in~\eqref{xH(x)-cb-deg-eq} and~\eqref{xH(x)-st-deg-eq}.
\par
The properties $\frac1{x^2H(x)}\in\mathcal{S}$ and $x^2H(x)\in\mathcal{CB}$ follow from~\eqref{frac1z2H(z)-int} and~\cite[Theorem~7.3]{Schilling-Song-Vondracek-2nd}. By~\eqref{zH(z)-int-exp}, we have
\begin{equation*}
z^2H(z)=\int_0^\infty\rho(t)\frac{z}{t+z}\td t
=\int_0^\infty\biggl[\int_0^\infty{t\rho(t)}e^{-tu}\td t\biggl]\bigl(1-e^{-zu}\bigr)\td u.
\end{equation*}
L\'evy-Khintchine representation~\eqref{z^2H(z)-int-eq} is established.
\par
The degrees appeared in~\eqref{x^2H(x)-cb-deg-eq} and~\eqref{x^2H(x)-st-deg-eq} may be calculated by the same argument as in the proof of the degrees in~\eqref{xH(x)-cb-deg-eq}.
\par
By virtue of the second inclusion in~\eqref{S-L-C-relation}, the logarithmically complete monotonicity $xH(x)\in\mathcal{L}[(0,\infty)]$ and $\frac1{x^2H(x)}\in\mathcal{L}[(0,\infty)]$ may be derived from the properties $xH(x)\in\mathcal{S}$ and $\frac1{x^2H(x)}\in\mathcal{S}$ respectively.
\par
All of positive operator monotonicity may be deduced from the property that they are complete Bernstein functions by available of~\cite[Theorem~12.17]{Schilling-Song-Vondracek-2nd} recited on page~\pageref{Theorem.12.17-SV2} and before Definition~\ref{deg-stieltjes-dfn}.
The proof of Theorem~\ref{third-conj-thm-conseq} is complete.
\end{proof}

\begin{rem}
This paper is a slightly revised version of the preprints~\cite{frac-lnz-conj-corr.tex, frac-lnz-conj-corr.tex-Dataset}.
\end{rem}

\end{document}